\documentclass{amsart}
\usepackage{graphicx} 
\usepackage{amsfonts,amsmath,tikz,mathabx,caption}
\usepackage{amssymb}
\usepackage{hyperref}
\usepackage{pb-diagram}
\usepackage[all]{xy}
\usepackage{hyperref,xcolor}
\usetikzlibrary{decorations.markings}
\usetikzlibrary{shapes}
\usetikzlibrary{backgrounds}
\usepackage{subfig}

  \newtheorem{thm}{Theorem}
\newtheorem{lemma}[thm]{Lemma}

\newtheorem{rmk}[thm]{Remark}
\newtheorem{cor}[thm]{Corollary}
\newtheorem{prop}[thm]{Proposition}

\numberwithin{thm}{section}

\newcommand{\gam}{\Gamma}
\newcommand{\Z}{\mathbb{Z}}
\newcommand{\N}{\mathbb{N}}
\newcommand{\Q}{\mathbb{Q}}

\DeclareMathOperator{\Homeo}{Homeo}

\DeclareMathOperator{\ro}{RO}

\DeclareMathOperator{\supp}{supp }
\DeclareMathOperator{\suppe}{supp^e}

\newcommand{\GG}{\mathcal{G}}
  
\title{Generic torsion-free groups and Rubin actions}
\author{Thomas Koberda and Yash Lodha}
\date{\today}

\begin{document}

\begin{abstract}
    We use model theoretic forcing to prove that a generic 
    countable torsion-free group
    does not admit any non-trivial locally moving action on a
    Hausdorff topological space, and yet admits a rich Rubin poset. 
\end{abstract}

\maketitle

\section{Introduction}

In this article, we investigate the question of whether a generic countable
torsion-free group admits
sufficiently rich actions on topological spaces. We are motivated primarily by
the problem of deciding whether or not there exists a countable torsion-free
group which admits no non-trivial action on a compact manifold;
for instance, countable groups acting faithfully
on connected compact $1$--manifolds are
characterized by linear or circular orderability. For manifolds of
dimension two or higher, there are no known obstructions for
a torsion-free countable group to act faithfully by homeomorphisms, though
it seems likely that there are countable torsion-free groups which admit
no non-trivial action by homeomorphisms on any compact manifold.

Let $X$ be a Hausdorff topological space, and let $G\leq\Homeo(X)$ be a subgroup.
For $U\subseteq X$ open, we write $G_U$ for the \emph{rigid stabilizer} of $U$,
which is to say the subgroup of $G$ consisting of elements which restrict to the
identity outside of $U$. We say that the action of $G$ is: 
\begin{enumerate}
\item \emph{Locally dense} if for all
$U\subseteq X$ open and for all $p\in U$, we have that the closure of the orbit
$G_U\cdot p$ has nonempty interior. 
\item \emph{Locally moving} if for all nonempty $U\subseteq X$, we have $G_U$ is
non-trivial. 
\end{enumerate}
It is generally the case (though not always, like in the case of a manifold with
boundary)
that requiring an action to be locally moving is weaker than
requiring it to be locally dense.
Locally dense actions of groups find their
importance through the following fundamental result of 
Rubin~\cite{Rubin1989,Rubin1996,rubin-short,KK2021book}:
\begin{thm}\label{thm:rubin}
    Let $X$ and $Y$ be locally compact and Hausdorff topological spaces with no isolated points, and let $G$ be a group
    acting faithfully and locally densely on both $X$ and $Y$. Then there is a $G$-equivariant
    homeomorphism $X\longrightarrow Y$.
\end{thm}
A locally dense faithful action of a group on a locally compact
Hausdorff space $X$ with no isolated points will be called a
\emph{Rubin action}.
We will say that a group $G$ is a \emph{Rubin group} if $G$ admits a Rubin action, and a \emph{weakly Rubin group} 
if it admits a homomorphism to $\Homeo(X)$ for some Hausdorff space
$X$ with at least two points, whose image is a locally moving group of homeomorphisms. 

Let $G\leq\Homeo(X)$ be a Rubin group. The
driving force behind Rubin's Theorem (Theorem~\ref{thm:rubin} above) is that from the
local density condition, 
one can recover a substantial amount of the topology of $X$. Specifically, one
can recover a dense subset of the Boolean algebra
$\ro(X)$ of \emph{regular open sets} of $X$; here, a subset $D$ of $X$ is
dense if for every nonempty $U\in X$, there is a nonempty $V\in D$ such
that $V\subseteq U$.
This fact has been used to
investigate the model theory of homeomorphism groups of manifolds;
see~\cite{KKdlN22,KdlN2023,KdlN2024}.
Here,
we say an open set is \emph{regular} if it is equal to the interior of its closure.
Regular open sets in $X$ arise from supports of homeomorphisms of $X$.

A homeomorphism $g$ of $X$ has an \emph{open support} $\supp g$, which consists
of the points in $X$ which are not fixed by $g$. 
We define the \emph{extended support} of $g$, or $\suppe g$ as the interior of the closure of $\supp g$; this is the smallest regular
open set containing $\supp g$. 
We say that $g\in G$ is
\emph{algebraically disjoint} from $f\in G$ if for all $h\in G$ such that $[f,h]\neq 1$,
there are elements $a,b\in C_{G}(g)$ such that \[1\neq [a,[b,h]]\in C_{G}(g).\]
Here, $C_{G}(g)$ is the centralizer of $g$ in $G$; note
that algebraic disjointness makes sense in an arbitrary group and not
just in a homeomorphism group. Also, note that if $g$ is algebraically disjoint from $f$,
then $[g,f]=1$, for indeed choosing $h=g$ provides a contradiction.
It is a trivial though useful observation that in an abelian group, two
elements are algebraically disjoint if and only if they commute, since
the quantification over $h$ not commuting with $f$ is vacuous.

Algebraic disjointness is not necessarily
a symmetric relation; for example, if $G$ is the symmetric group on $\{1,2,3,4\}$,
the permutation $(1\:2)$ is algebraically disjoint from $(1\: 2)(3\: 4)$, but not
vice-versa.

In the proof of Rubin's theorem, a partially ordered set is constructed purely from the algebraic structure of the group,
and this poset is shown to be isomorphic to inclusion ordered poset of finite intersections $$\{\suppe(g_1)\cap \cdots \cap \suppe(g_n)\mid g_1,\ldots,g_n\in G, n\in \N,n\geq 1\}$$
We recall the algebraic construction.
For $f\in G$, let $$S_f=\{g^{12}\mid g\in G\text{ algebraically disjoint from }f\}$$
and let $C_G(S_f)$ be the centralizer of $S_f$ in $G$.
The group $C_G(S_f)$ is non-trivial since it follows from the
definition of algebraic disjointness that $f\in C_G(S_f)$.
Also, note that it is possible that $C_G(S_f)=C_G(S_{f'})$ for distinct elements $f,f'\in G$.
Let $\mathcal{P}(G)$ be the poset consisting of elements $$\{C_G(S_{f_1})\cap \cdots \cap C_G(S_{f_n})\mid f_1,\ldots, f_n\in G, \,1\leq n\in \N\},$$ and partially ordered by inclusion. In the presence of a Rubin action on a space $X$, the poset $\mathcal{P}(G)$ is naturally isomorphic to the 
inclusion ordered poset of extended supports of elements of $G$ in $X$
~\cite{Rubin1996,rubin-short,KK2021book}.
Indeed, the map $C_G(S_f)\mapsto \suppe(f)$ defined for all
$f\in G$ extends to such an isomorphism.

The main result of this paper is that ``most" countable groups are not Rubin groups (or even weakly Rubin groups), yet admit a rich Rubin poset.
Here,
countable groups are organized into a Polish space of marked countable groups,
which can then be investigated from a descriptive set theory point of view; see
~\cite{GKL2023}. In
particular, the notions of meagerness and comeagerness make sense for the space
of marked countable groups, and a property of groups is called \emph{generic} if it
holds for groups in a comeager subset of the space of marked countable groups.
Properties of generic groups
are studied by certain (infinite length) two-player games, and the class of groups obtained
at the ``end" of the game is called the class of \emph{compiled groups};
we will give a precise definitions in Section~\ref{sec:background}.
The main result of this article is:

\begin{thm}\label{thm:main}
Let $G$ be a generic torsion-free countable group.
\begin{enumerate}
    \item $G$ does not admit a Rubin group or a weakly Rubin group as a quotient.
    \item The Rubin poset $\mathcal{P}(G)$ coincides with the poset of cyclic 
    subgroups of $G$, ordered by inclusion, and
    contains both
bi-infinite chains and infinite antichains.
\item Two non-trivial elements of $G$ commute if and only if they are algebraically
disjoint in $G$.
\end{enumerate}
\end{thm}

It is easy to find groups which do not admit weakly Rubin actions on
Hausdorff topological spaces with at least two points; free groups are one
such example, since no two elements in a free group can generate a copy of
$\Z^2$. The content of Theorem~\ref{thm:main} is two-fold: most torsion-free groups
cannot have weakly Rubin actions, but look like they should.

\begin{rmk}
  A generic torsion-free group has a rich subgroup structure. 
In fact, it contains every finitely generated torsion-free group with a solvable word problem as a subgroup.
In turn, the poset of cyclic subgroups ordered by inclusion is very rich: for
instance, a generic torsion-free group will contain a copy of $\Q*_{\Z}\Q$,
the amalgamated product of two copies of $\Q$ over their respective copies of
the integers. We thus obtain many cyclic subgroups which are all distinct, but
which contain a common cyclic subgroup. One can repeat a process of
this amalgamation and taking direct sums of groups
\emph{ad infinitum}, thus building a small piece of the very complicated
poset of cyclic subgroups of a generic torsion-free group.
Also, we remark that any comeager subset of the space of torsion-free groups
will contain continuum many groups up to isomorphism (this is a direct consequence of Theorem $1.1.6$ of \cite{GKL2023}).  
\end{rmk}

Theorem~\ref{thm:main} is established through model theoretic forcing; the
key is to find a
Banach--Mazur game in which one player can force any pair of commuting elements
which do not share a common power to be mutually algebraically disjoint. The same player
also forces any two non-trivial elements to be conjugate, resulting in simple
compiled groups.
We also emphasize the fact that our result is much easier to prove if the torsion-free requirement is dropped, though the requirement that the group be torsion-free is central to our program of investigating obstructions to group actions on compact manifolds:
indeed, the existence of torsion is often an elementary source of such obstructions, and should be excluded for the development of a deeper theory. 

\section{Background}\label{sec:background}

In this section, we recall some basic notions of Rubin's theory of group actions
on topological spaces, and model theoretic forcing.

\subsection{Algebraic disjointness}
Let $\Gamma\leq\Homeo(X)$ be a Rubin action on a Hausdorff topological space $X$.
The relationship
between algebraic disjointness and actual disjointness of supports is
given by the following:

\begin{prop}[See~\cite{Rubin1996,rubin-short,KK2021book}]\label{prop:disjoint}
    Let $\Gamma\leq\Homeo(X)$ be a locally moving action.
    \begin{enumerate}
        \item If $f,g\in\Gamma$ satisfy $(\supp g)\cap(\supp f)=\varnothing$ then
        $f$ and $g$ are mutually algebraically disjoint.
        \item If $f,g\in\Gamma$ are such that $g$ is
    algebraically disjoint from $f$, then $(\supp g^{12})\cap(\supp f)=\varnothing$.
    \end{enumerate}
\end{prop}

\subsection{The space of enumerated groups}
A countably infinite group $G$ can be abstractly identified with the natural numbers
$\N$. The multiplication operation is then a map $\N\times\N\longrightarrow \N$,
and the inversion map is just a map $\N\longrightarrow\N$. Countable groups can
be viewed as a subspace of the Polish (i.e.~separable and completely metrizable) space
\[X=\N^{\N\times\N}\times\N^{\N}\times\N,\]
corresponding to the choice of multiplication and inversion functions, and the
unique identity element; this space is given the natural product topology.
The subspace $\GG\subseteq X$ corresponding to groups is
closed in this topology and is therefore a Polish space in its own right;
we discuss this further below.  We remark that spaces of countable models of a theory are of
general interest in infinitary logic and descriptive set theory;
see~\cite{marker-infinitary}.

Let $\underline x$ denote a finite tuple of variables.
We let $\Sigma(\underline x)$ be a quantifier-free formula in the language of group
theory consisting of finitely many equations and inequations of the form
$w(\underline x)=1$ or $w(\underline x)\neq 1$. Such a $\Sigma(\underline x)$
is called a \emph{finite system}. If $\underline g$ is a finite tuple of natural numbers, the system $\Sigma(\underline g)$ defines a clopen set $U_{\Sigma}$
in $\GG$ by considering
the groups in which $\Sigma(\underline g)$ holds, and these clopen sets determine a basis for the topology of $\GG$.

The preceding remarks have the important consequence
that the subspace of countable torsion-free groups
$\GG_{\mathrm{tf}}$ is closed in $\GG$. Indeed, the
requirement that a particular element of a group has a fixed finite order is an
open condition, and so the property of being torsion-free is a closed
condition; cf.~\cite{GKL2023}.
It follows that $\GG_{\mathrm{tf}}$ is a Polish subspace and the intersection of the aforementioned clopen sets with $\GG_{\mathrm{tf}}$ provides a 
basis of clopen sets for $\GG_{\mathrm{tf}}$. Since $\GG_{\mathrm{tf}}$ is Polish,
the Baire Category Theorem holds:
a countable intersection of open dense sets
in $\GG_{\mathrm{tf}}$ is dense.

\subsection{Banach--Mazur games}

Let $P$ be a property of countable groups, and let $\GG_P$ be the subset of
$\GG$ consisting of groups satisfying $P$. A \emph{Banach--Mazur game}~\cite{hodges-book,marker-infinitary}
is a game in which players A and B take turns
choosing finite systems $\{\Sigma_i(\underline{g})\}_{i\in\N}$ such that the nonempty clopen sets
$U_i=U_{\Sigma_i}$ are nested and whose intersection determines a class
of groups in $\GG$ called the \emph{compiled groups} of the game. The nestedness of these
open sets is crucial, since it prevents players from placing inconsistent conditions
on the compiled groups.
Any group
$G\in\bigcap_i U_i$ is called a \emph{compiled group} for the game.
One
should imagine that $A$ and $B$ are building a (class of groups) together,
by gradually naming the elements of the group and expressing, finitely many
equations
and inequations at a time, the (partial) multiplication table of a (class
of) groups. It may be the case that the full multiplication table can
be determined over the course of a game, in which case the compiled group
is unique. The nonemptiness of the sets $\{U_i\}_{i\in\N}$ is simply
expressing the consistency of the finite systems played up to that round
of the game; in particular, at least one group satisfies the conditions
forced at any round.

At each stage, both $A$
and $B$ may only play a finite system (and in particular not a family
of finite systems). On any turn, either player may play (any subset of)
the finite systems that have been played up to that round; this amounts
to that player wasting a turn. For the sake of convenience, we
assume that after $n$ rounds, all natural numbers below $2n$ have been
played (and possibly more).

Player A \emph{wins} if 
the compiled groups lie in $\GG_P$. 
Player A has a \emph{winning strategy} if for
any possible
position $U_n$ in the game, there is a possible choice of systems for which
player A wins, regardless of how player B plays; the choice of systems
played by player A need not be fixed beforehand, but can depend on exactly
how player B proceeds.
It is a standard fact that player A has a
winning strategy if and only if $\GG_P$ is \emph{comeager},
which is to say that $\GG_P$ contains a countable intersection of dense open sets.
We say that $P$ is a \emph{generic} property of groups if this holds.
Note that a countably infinite conjunction of generic properties (i.e.~simultaneously requiring all of the countably many properties) is generic, from the Baire category theorem; indeed, the set of groups with each of the generic
properties contains a countable intersection of dense open sets, and so the set of groups with countably many prescribed generic properties also contains a countable intersection of dense open sets. 

Similarly, player B wins if the compiled groups lie in
the complement of $\GG_P$, and has a winning strategy if for any
position $U_n$ in the game, there is a possible choice of systems for which
player B wins. For instance, suppose $P$ is the property of having at least
two conjugacy classes (in a torsion-free group). Without loss of generality,
$0$ has been declared to be the identity by player A in the first round.
At every turn, player
B chooses an $n\in\N$ exceeding the maximum of the numbers that have
been played by either player, declares $n$ to be distinct from $\{0,\ldots,
n-1\}$, and declares that $n$ conjugates $k$ to $k+1$ for $k\geq 1$.
This is easily seen to be a winning strategy for player B.

\subsection{Divisible groups}
In later parts of this paper, \emph{divisible groups} will arise.
A general reference on divisible groups is~\cite{kaplansky-book}.
Divisible
groups are usually (though not always, like in this paper)
abelian groups $A$ wherein for all $x\in A$, the equation
$n\cdot y=x$ has a solution $y\in A$. There is an obvious generalization
to nonabelian groups, wherein one requires solutions to the equation
$y^n=x$. To allow for torsion-free divisible groups, one need not require
$y$ to be non-trivial.

Abelian divisible groups have a well-defined torsion subgroup. An abelian
torsion-free
divisible group is isomorphic to a Cartesian power of the rationals,
say $\Q^I$. The cardinality of the set $I$ is called the \emph{rank} of
the divisible group.

\section{Some combinatorial group theory}\label{sec:combinatorial group theory}

This section contains a technical foray into combinatorial group theory,
which will be crucial for setting up a suitable Banach--Mazur game in
Section~\ref{sec:forcing}. We will use standard ideas from combinatorial group
theory, which could be found in~\cite{Serre1977,LS2001}, for instance.

We will always assume that abstract relations
in a group are reduced and cyclically reduced. Moreover, when
dealing with a group presentation, we will always assume that the set of
relations is closed under taking inverses and cyclic permutations.
Our standing notation will be that $G=\langle g,h\mid\mathcal R_0\rangle$ is
a subgroup of larger ambient torsion-free group $\langle f,g,h\rangle$, wherein $[f,h]\neq 1$ but $[f,g]=1$.
Note that this implies that $h\notin \langle g\rangle$ and that if $G$ is cyclic then $g=h^n$ for some
    $n\in \mathbb{Z}$ with $|n|>1$.

\begin{lemma}\label{lem:overgroup}
    Let the group $G$ defined
    above be non-trivial and torsion-free group. Suppose furthermore that
    $g,h\neq 1$, and if $G$ is cyclic then $g=h^n$ for some
    $n\in \mathbb{Z}$ with $|n|>1$.
    Then there exists a finitely generated torsion-free group $\gam$
    such that:
    \begin{enumerate}
        \item $G$ embeds as a subgroup of $\gam$.
        \item There exist distinct elements
        $a,b\in\gam\setminus 1$ centralizing $g\in G$.
        \item The commutator $[a,[b,h]]\neq 1$ and centralizes $g$.
    \end{enumerate}
\end{lemma}

We will prove Lemma~\ref{lem:overgroup} in both the cyclic and non-cyclic
cases below within the body of the proof, though in order to prove
Lemma~\ref{lem:overgroup} in the
non-cyclic case, we will need to establish the following lemma.

\begin{lemma}\label{lem:normal-closure}
    Let $K=\langle g,\gamma\rangle$ be a non-trivial and torsion-free group, and suppose that for all nonzero
    $n\in\Z$ we have $\gamma^n\notin \langle\langle g\rangle\rangle_K$. Then there is
    a finitely generated torsion-free overgroup $\gam$ and a non-trivial
    $a\in\gam$ such that:
    \begin{enumerate}
        \item $K$ embeds as a subgroup of $\gam$.
        \item The commutator $[a,\gamma]$ is non-trivial and centralizes $g$.
        \item $[a,g]=1$.
    \end{enumerate}
\end{lemma}
\begin{proof}
    Let $K_1$ and $K_2$ be two isomorphic copies of $K$, with $g_i,\gamma_i\in K_i, 1\leq i\leq 2$
    being the corresponding generators in the two copies. Set
    \[\gam=K_1\times (K_2*\Z),\] and let \[K_3=\langle (g_1,1),(\gamma_1,\gamma_2)\rangle.\]
    Clearly, $\gam$ is torsion-free.
    
    Observe that $K_3$ is isomorphic to $K$ via the map $\phi$ sending
    \[g\mapsto (g_1,1)\quad \gamma\mapsto (\gamma_1,\gamma_2).\] Indeed,
    if the map $\phi$ is a well-defined homomorphism
    of groups then it is clearly surjective; moreover,
    by projection onto
    the first factor, we see that it must be injective as well. Thus, it
    suffices to show that if $w(g,\gamma)$ is a relation in $g$ and $\gamma$ then
    $\phi(w(g,\gamma))$ is the identity. By the assumption that
    $w(g,\gamma)$ is a relation in $g$ and $\gamma$ and the definition of
    $\phi$, we see that
    $\phi(w(g,\gamma))$ will be the identity if and only if $w$ has zero exponent
    sum in $\gamma$.
     Since we assumed that
    $\gamma^n\notin \langle\langle g\rangle\rangle_K$ for all nonzero $n$,
    it is immediate that any relation in $K$ will have zero exponent sum in $\gamma$.

    It thus suffices to find the desired element $a\in C_{\gam}(g)$; we simply
    take a generator $a\in\Z$ of the infinite cyclic group in the free product.
    By construction, $(1,a)$ commutes with $(g_1,1)\in K_3$, and the commutator
    $[(1,a),(\gamma_1,\gamma_2)]$ is non-trivial and commutes with $(g_1,1)$.
\end{proof}

Our goal is to show that Lemma~\ref{lem:overgroup} can
be reduced to the hypotheses of Lemma~\ref{lem:normal-closure}.

\begin{lemma}\label{lem:amalgamation}
    Let $G=\langle g,h\rangle$ be a non-cyclic and torsion-free group, and let $H$ be a free product with amalgamation of
    two copies $G_1=\langle g_1,h_1\rangle$ and $G_2=\langle g_2,h_2\rangle$ of $G$, wherein the subgroups $\langle g_1\rangle$
    and $\langle g_2\rangle$ being amalgamated into a single cyclic group
    $\langle g\rangle$. Consider now the subgroup
    $K=\langle h_2^{-1}h_1, g\rangle$. Then for all nonzero $n\in\Z$, we have
    \[(h_2^{-1}h_1)^n\notin \langle\langle g\rangle\rangle_{K}.\]
\end{lemma}

Assuming Lemma~\ref{lem:amalgamation}, we obtain Lemma~\ref{lem:overgroup}:

\begin{proof}[Proof of Lemma~\ref{lem:overgroup}]
    We handle the case of $G$ cyclic, and then the case of $G$ non-cyclic.
   In the case where $G$ is cyclic, then it must be the case that for some $n\in \mathbb{Z}$ with $|n|>1$ we have that $h^n=g$.
    Indeed, if $g^n=h$ then $h$ must commute with $f$ which contradicts our assumption.
    We may produce the required overgroup by amalgamating a copy of $\mathbb{Z}^3$ generated by elements $\{a,b,c\}$
    with $G$, by identifying the cyclic subgroups
    $\langle c\rangle$ and $\langle g\rangle$.
    Using the normal form for amalgamated free products, it is
    clear that $[a,[b,h]]\neq 1$.
    Moreover, \[[[a,[b,h]],g]=1,\] since $[\alpha,g]=1$ for each
    $\alpha\in \{a,b,h\}$.
    
    Assume now that $G$ is non-cyclic.
    Let $Q$ be the group $H$ obtained by amalgamating two copies of $G$ along
    their respective copies of $g$, and let $\beta$ be the automorphism of
    $Q$ which fixes $g$ and exchanges the two copies $h_1$ and $h_2$
    of $h$ in $Q$, coming from
    the two amalgamated subgroups. We set $\gam_0$ to be the semidirect product
    of $Q$ with $\Z$, where $\Z$ acts on $Q$ by $\beta$. The group
    $Q$ is torsion-free by standard methods; see the discussion of free products with amalgamation in \cite{Serre1977}, for instance. Therefore
    $\gam_0$ is
    torsion-free as well.

    We identify the generator of $\Z$ as an element $b$ of $\gam_0$, and
    we note that $b$ centralizes $g$. Observe
    that $[b,h_1]=h_2^{-1}h_1$; we call this commutator $\gamma$, which
    clearly represents a non-trivial element of $\gam_0$.

Let $K=\langle g,\gamma\rangle$. By Lemma~\ref{lem:amalgamation}, we have that
$\gamma^n\notin \langle\langle g\rangle\rangle_{K}$ for any $n\in \mathbb{Z}\setminus \{0\}$, and so by
Lemma~\ref{lem:normal-closure}, the group $K$ embeds in a
torsion-free overgroup $\tilde K$, wherein some non-trivial $a\in \tilde K$
centralizes $g$ and satisfies $[a,\gamma]\neq 1$ and $[a,\gamma]$ centralizes
$g$. Now, amalgamate
$\tilde K$ and $\gam_0$ over their respective copies of $K$, and call
the resulting group $\gam$. In $\gam$, we have \[1\neq [a,\gamma]=[a,[b,h_1]],\]
and all three of $a$ and $b$ and $[a,[b,h_1]]$ centralize $g$.
Since $\langle g,h_1\rangle$ is an isomorphic copy of $G$, we are done.
\end{proof}

We now need only prove Lemma~\ref{lem:amalgamation}. For this, we use Van
Kampen diagrams over $H$; a reader unfamiliar with Van Kampen diagrams and
disk diagrams
may consult Chapter 5 of~\cite{LS2001} or~\cite{olshanski-book}, for example.
Of course, Van Kampen diagrams will generally depend
on a particular choice of presentation. When we construct $H$, we have
two copies of $G$, and as such there are two copies of the set of relations
$\mathcal R_0$, which we call $\mathcal R_1$ and $\mathcal R_2$, respectively.
Thus, relations in $\mathcal R_i$ involve $g=g_1=g_2$ and $h_i$ only. We fix
the resulting presentation for $H$ once and for all.

We make some observations about several special kinds of relations in $G$.

\begin{lemma}\label{lem:baumslag-solitar}
    Suppose that in $G$ there are nonzero integers $n,m\in\Z$ such that
    $hg^nh^{-1}=g^m$. Then in the group $H$, we have $[g^n,h_2^{-1}h_1]=1$.
\end{lemma}
\begin{proof}
    This follows immediately from computing the relevant commutator.
\end{proof}

In a Van Kampen diagram $\Delta$ over $H$, we call cells in the interior
of $\Delta$ \emph{tiles}. Tiles coming from relations in $\mathcal R_i$ will
be called \emph{$i$--tiles}, for $i\in\{1,2\}$. A maximal connected
(without cut vertices) union
$\Xi$ of $i$--tiles for
a fixed value of $i$
will be called an \emph{$i$--region}. Observe that if a $1$--tile shares
an edge with a $2$--tile then the label of that edge is $g$.

If $\Delta$ is a \emph{disk diagram} over $H$ (i.e.~a Van Kampen diagram
with no cut edges and no cut vertices) and if $\Xi$ is an $i$--region, then
we will say that $\Xi$ is \emph{separating} if the complement of $\Xi$ in
$\Delta$ contains at least two components $\Theta_1$ and $\Theta_2$ which meet
$\partial\Delta$; even if $\Xi$ is non separating, $\Xi$ itself may fail to be
simply connected and may still topologically separate $\Delta$.

Since $\Delta$ is homeomorphic to a disk, an easy combinatorial
topology argument shows that $\Delta$ admits at least one non-separating
$i$--region, for some $i\in\{1,2\}$. Indeed, let $\Delta_0$ be an $i$-region
in $\Delta$. Note that $\Delta_0$ must share an edge with $\partial\Delta$,
since any edge shared by a $1$--region and a $2$--region is labeled by $g$;
thus, if $\Delta_0$ were contained entirely in the interior of $\Delta$ then
the boundary label of $\Delta_0$ would be a power of $g$. This power of $g$
cannot be nonzero since $G$ is torsion-free and non-trivial. If the power of
$g$ were zero then we could cut $\Delta_0$ out of $\Delta$ and identify
edges in $\partial\Delta_0$ to get a smaller area (i.e.~fewer tiles)
disk diagram $\Delta'$ with
the same boundary label as $\Delta$.

Writing $\Delta$ as a union $\mathcal S$ of distinct $1$--regions and
$\mathcal T$ of distinct $2$--regions, 
we thus conclude that each $\Omega\in \mathcal S$ either separates $\Delta$ into
two components, or $\Omega$ is non-separating. Every element of $\mathcal S\setminus\Omega$ and every element
of $\mathcal T$ must lie in a component of $\Delta\setminus \Omega$; a similar
argument holds for $\Omega\in\mathcal T$. By
an easy induction, we conclude that there is a non-separating $i$--region
for some $i$.

Note that if $\Xi$ is a non-separating $i$--region
in $\Delta$ then $\Xi$ meets $\partial\Delta$ in an arc. With these
observations, we are ready to proceed.

\begin{proof}[Proof of Lemma~\ref{lem:amalgamation}]
    Assume the contrary, and let $w'(g,h_2^{-1}h_1)$ be a word
    that expresses a nonzero power
    $(h_2^{-1}h_1)^n$ as a product of conjugates of $g$ and its inverse by elements of $K$.
    Observe that the expression $(h_2^{-1}h_1)^n=w'(g,h_2^{-1}h_1)$ is
    equivalent to some other word $w(g,h_2^{-1}h_1)$ representing
    the identity in $H$ and having exponent sum $n$ in $h_2^{-1}h_1$.
    We therefore let $\Delta$ be a reduced minimal area Van Kampen diagram
    for a word $w(g,h_2^{-1}h_1)$ which has nonzero exponent sum in
    $h_2^{-1}h_1$, and which has a
    minimal
    number of occurrences of $h_2^{-1}h_1$.
    Generally, $\Delta$ may not be a disk
    diagram and so may have separating vertices and edges, though there will
    be at least one positive area disk diagram $\Delta_0\subseteq\Delta$.

    Observe first that $\partial\Delta_0$ must have at least one occurrence
    of $h_2^{-1}h_1$, since otherwise $\Delta_0$ proves that $g$ has finite
    order in $H$, which is not the case. Choose a non-separating $i$--region
    $\Xi$, which we may assume without loss of generality is a $1$--region.
    Observe that the label of the arc $\partial\Xi\cap\partial\Delta_0$ must
    be of the form \[h_1^{\mp 1}g^nh_1^{\pm 1}\quad\textrm{or}
    \quad h_1^{\pm 1}g^nh_1^{\pm 1} \quad\textrm{or}\quad h_1^{\pm1}g^n\] for some
    nonzero value of $n$. Indeed, otherwise the
    label of the arc $\partial\Xi\cap\partial\Delta_0$ must
    be $g^n$; since the remainder of $\partial\Xi$ coincides with arcs in the
    boundaries of $2$--regions, we conclude that $\partial\Xi$ is a power of
    $g$, which is not the case since $G$ is torsion-free and because we assumed
    the set $\mathcal R$ to consist of reduced and cyclically reduced words.

    Observe furthermore that $\partial\Xi\cap\partial\Delta_0$ cannot read
    $h_1^{\pm 1}g^nh_1^{\pm 1}$. Indeed, since the boundary of $\Delta$ is
    a word in $g$ and $h_2^{-1}h_1$, a negative power of $h_1$ must be
    immediately followed by a positive power of $h_2$ and a positive power
    of $h_1$ must be immediately preceded by a negative power of $h_2$. It
    also cannot be the case that $\partial\Xi\cap\partial\Delta_0$ reads
    $h_1^{\pm1}g^n$, since in this case we see that $h$ coincides with a power
    of $g$, a contradiction.
    
    It follows that $\partial \Xi$ reads
    $h_1g^nh_1^{-1}g^m$ for some suitable
    value of $m$. We must have $m\neq 0$ since otherwise we obtain $g^n=1$ in $G$. Then,
    by Lemma~\ref{lem:baumslag-solitar}, we see that $h_2^{-1}h_1$ commutes
    with $g^n$ in $H$. Now, since the boundary of $\Delta_0$ is a word
    $w_0(g,h_2^{-1}h_1)$, the arc $\partial\Xi\cap\partial\Delta_0$ lies in a
    larger arc whose boundary reads \[h_2^{-1}h_1g^nh_1^{-1}h_2;\]
    we may therefore
    decrease a pair of occurrences of $h_2^{-1}h_1, h_1^{-1}h_2$
    in the boundary word $w$ of
    $\Delta$, while maintaining its exponent sum in $h_2^{-1}h_1$. This violates
    the minimality of the choice of $w$.
\end{proof}


\section{Forcing algebraic disjointness through Banach--Mazur 
games}\label{sec:forcing}

In this section we study generic torsion-free countable groups.
In other words, we study comeager subsets of the space of enumerated torsion-free groups. A general reference for this section is~\cite{hodges-book}.

Recall that a \emph{generic property}
in this space is a property for which there is a comeager
subspace in which all groups satisfy the property.
The following is direct consequence of Theorems 1.2.1. and 1.1.2 in ~\cite{GKL2023}; 
since a countable conjunction of generic properties is again a
generic property, in this section we shall assume that a generic torsion-free group satisfies all the properties listed in this theorem:

\begin{thm}\label{thm:forcingGKL}
There is a comeager set $\mathcal X$
in the space of enumerated torsion-free groups such that
for all $G\in \mathcal X$:
\begin{enumerate}
    \item All non-trivial elements $1\neq g\in G$ are conjugate.
    In particular, $G$ is simple.
    \item If $H$ is a finitely generated torsion-free group with a solvable
    word problem then the group $G$ admits $H$ as a subgroup.
    \item The group $G$ is verbally complete.
\end{enumerate}
\end{thm}

In Theorem~\ref{thm:forcingGKL}, a group
$G$ is called \emph{verbally complete} if for
every reduced word $w$ in the free group on $k$ generators (for $k\in\N$ arbitrary) and 
all $g\in G$, the equation $w(x_1,\ldots,x_k)=g$ admits a solution in $G$.

In this section, we prove the following fact:
\begin{prop}\label{prop:forcing}
    Let $G$ be a generic countable torsion--free group. Then for all pairs
    of non-trivial elements $f,g\in G$ such that $[f,g]=1$, we have that
    $f$ is algebraically disjoint from $g$. In particular,
    $S_f=C_G(f)$ for each $f\in G$.
\end{prop}

\begin{proof}[Proof of Proposition~\ref{prop:forcing}]
    We play a Banach--Mazur game on the space of countable torsion-free groups. In order to
    prove that a generic group $G$ has the desired properties, it suffices to prove
    that the first player (player A) has a winning strategy which forces all the compiled groups to satisfy that for all pairs
    of non-trivial elements $f,g\in G$ such that $[f,g]=1$, we have that
    $f$ is algebraically disjoint from $g$.
    Since the space of countable
    torsion-free groups is closed and therefore Polish, it suffices to show that
    at any stage of the game, the conditions played so far are compatible with
    torsion-freeness.

    In a given stage of the game, a move (i.e.~a finite system of equations and inequations) is called \emph{admissible} if
    it is consistent with the moves played so far. In other words, there is an enumerated torsion-free group which witnesses the union of the posited
    equations and inequations.
    
    An equation $w(\underline k)=1$ or an inequation $w(\underline k)\neq 1$ is said to be \emph{forced} at stage $n$ if it is a formal consequence
    of the 
    union of moves played until stage $n$. This means that no matter what moves are played for the rest of the game, any compiled group
    will satisfy the forced condition.
    Thus, if an equation (or inequation) and its negation has not been forced at stage $n$, this means that it is admissible and therefore
    can be played legally as part of a move at stage $n+1$. 

    Let \[\{(f_n,g_n,h_n)\mid n\in \mathbb{N}\}\] be an enumeration of all $3$-tuples in $\mathbb{N}\times \mathbb{N}\times \mathbb{N}$.
    In the first move, player A fixes $1$ to be the identity element.
    In the $n^{th}$ turn of player A, which is the $(2n-1)^{st}$ turn overall, A does the following: let $\Lambda$ be an instance of a compiled group that 
    satisfies all the moves played until the $(2n-2)^{nd}$ turn.
    \begin{enumerate}
    \item[Case 1]: It has been forced by now that $[f_n,g_n]=1$. We consider the following subcases:
    \begin{itemize}
    \item[1.1]: It has been forced by now that $[f_n,h_n]=1$. In this case, player A plays an empty system as a move.
    \item[1.2]: It has been forced by now that $[f_n,h_n]\neq 1$. Note that in this case, $h_n$ cannot be a power of $g_n$ in $\Lambda$. Let $a,b$ be numbers that have not appeared in a move so far. 
    In this case, player $A$ plays the following system: 
    $$[g_n,a]=1\qquad [g_n,b]=1$$ $$[a,[b,h_n]]\neq 1\qquad  [g_n,[a,[b,h_n]]]=1$$
    Note that this is admissible thanks to Lemma~\ref{lem:overgroup}, since an appropriate enumeration of the amalgamated free product of $\Lambda$ with the group $\Gamma$ from Lemma~\ref{lem:overgroup} over the subgroup $G=\langle g_n,h_n\rangle$ is a torsion-free group that witnesses the union of the systems played so far, including this move. 
    \item[1.3]: If neither $[f_n,h_n]=1$ nor $[f_n,h_n]\neq 1$ has been forced so far, player A plays $[f_n,h_n]=1$.
    \end{itemize}
    \item[Case 2]: It has been forced by now that $[f_n,g_n]\neq 1$.
    In this case, player A plays an empty system as a move.

    \item[Case 3]: If neither $[f_n,g_n]\neq 1$ nor $[f_n,g_n]=1$ has been forced so far, player A plays $[f_n,g_n]\neq 1$.
    
    \end{enumerate}

    That all the compiled groups will have the desired properties is now immediate.
\end{proof}


By modifying the Banach--Mazur game slightly, we can also guarantee:

\begin{cor}\label{cor:chain}
    For all non-trivial $g\in G$, we have $C_G(S_g)=\langle g\rangle$. In particular, for each $g_1,\ldots,g_n\in G$ with  $n\geq 1$,
    it holds that \[C_G(S_{g_1})\cap \cdots \cap C_G(S_{g_n})=\bigcap_{1\leq i\leq n}\langle g_i\rangle,\] which is a trivial or infinite cyclic group.
    Finally,
    $G$ is divisible. 
\end{cor}
\begin{proof}
    Divisibility is guaranteed by an application of Theorem \ref{thm:forcingGKL}.
    Now, let $g\in G$ be given, and let $1\neq k\in G$ be arbitrary element which
    is not a power of $g$. Then, player A builds an element $h$ that commutes with
    $g$ but no non-trivial power of $h$ commutes with $k$ (by building
    a suitable free product with amalgamation, for instance). This can be
    done at each stage of the game with all elements which have been played up to
    that point.
    Then,
    by Proposition~\ref{prop:forcing} we will have $h^{12}\in S_g$ and so $k\notin
    C_G(S_g)$. Since
    $g\in C_G(S_g)$ by definition, we see that $\langle g\rangle$ coincides with
    $C_G(S_g)$. The subsequent claim then follows from the fact that a finite intersection of infinite cyclic groups is trivial or infinite cyclic.
\end{proof}

\begin{cor}\label{cor:antichain}
    The Rubin poset of a generic countable torsion-free group will have
    bi-infinite chains and infinite antichains.
\end{cor}
\begin{proof}
    A generic countable torsion-free group
    $G$ will contain $\mathbb{Z}^2$ subgroups (thanks to Theorem~\ref{thm:forcingGKL}). In $\Z^2$, we can find an infinite sequence of distinct elements which pairwise
    generate $\Z^2$. 
    So any pair of
    these will be algebraically disjoint, and so these elements will give
    rise to a countably infinite antichain in the Rubin poset of $G$, from an application of Corollary \ref{cor:chain}.
    Finally, from an application of Theorem \ref{thm:forcingGKL} we know that a rank $1$ divisible group (such as the additive group of rational numbers) embeds in a generic group. We can find a bi-infinite inclusion ordered chain of cyclic subgroups in such a group, which applying Corollary \ref{cor:chain}, forms the required bi-infinite chain in the Rubin poset.
\end{proof}

\section{Proof of the main result}\label{sec:main}

The proof of the main result is now straightforward.
\begin{proof}[Proof of Theorem~\ref{thm:main}]
    Suppose that $G$ is a generic torsion-free countable group. By Proposition
    ~\ref{prop:forcing}, we may assume that
    all non-trivial elements of $G$ are conjugate,
    and that if $f,g\in G\setminus \{1\}$ commute then $f$ is algebraically disjoint from $g$. Suppose that
    \[G\longrightarrow \Homeo(X)\] is a homomorphism whose image $\Gamma$ is locally moving.
    Since $G$ is simple, we have that $G\cong \Gamma$, and so algebraic disjointness in
    $G$ propagates to algebraic disjointness in $\Gamma$. Since each non-trivial element $f$ is algebraically disjoint from itself,
    by Proposition~\ref{prop:disjoint}, the elements
    $f^{12}$ and $f^{12}$ would have disjoint supports
    as elements of $\Gamma$, which is impossible. 
    
    Another proof of this fact can be obtained by a forcing argument that ensures that in a generic torsion-free group
    for any two pairs $(f_1,g_1)$ and $(f_2,g_2)$ such that \[\langle f_i,g_i\rangle \cong \mathbb{Z}^2\] there exists $h\in G$ such that 
    $f_1^h=f_2$ and $g_1^h=g_2$.
    Then the proof is obtained by considering any pair $f,g\in G$ that have disjoint support (assuming $G$ admits a locally moving action on a Hausdorff space with at least two points),
    and observing that there is an element $h\in G$ such that $f^h=f$ and 
    $g^h=fg$, which is a contradiction since $f$ and 
    $fg$ do not have disjoint support.

    To see that the Rubin poset of the group $G$ consists of cyclic subgroups of
    $G$ ordered by inclusion with bi-infinite chains and infinite antichains,
    we simply quote Corollary~\ref{cor:antichain} and
    Corollary~\ref{cor:chain}. Proposition~\ref{prop:forcing} shows that
    two non-trivial elements of $G$ commute if and only if they are algebraically
    disjoint in $G$.
\end{proof}

\section*{Acknowledgments}
The first author is partially supported by NSF Grant
DMS-2349814, Simons Foundation International Grant SFI-MPS-SFM-00005890. 
The second author is supported by NSF Grant DMS-2240136.

\bibliographystyle{amsplain}
  \bibliography{ref}

\def\cprime{$'$} \def\soft#1{\leavevmode\setbox0=\hbox{h}\dimen7=\ht0\advance
  \dimen7 by-1ex\relax\if t#1\relax\rlap{\raise.6\dimen7
  \hbox{\kern.3ex\char'47}}#1\relax\else\if T#1\relax
  \rlap{\raise.5\dimen7\hbox{\kern1.3ex\char'47}}#1\relax \else\if
  d#1\relax\rlap{\raise.5\dimen7\hbox{\kern.9ex \char'47}}#1\relax\else\if
  D#1\relax\rlap{\raise.5\dimen7 \hbox{\kern1.4ex\char'47}}#1\relax\else\if
  l#1\relax \rlap{\raise.5\dimen7\hbox{\kern.4ex\char'47}}#1\relax \else\if
  L#1\relax\rlap{\raise.5\dimen7\hbox{\kern.7ex
  \char'47}}#1\relax\else\message{accent \string\soft \space #1 not
  defined!}#1\relax\fi\fi\fi\fi\fi\fi}
\providecommand{\bysame}{\leavevmode\hbox to3em{\hrulefill}\thinspace}
\providecommand{\MR}{\relax\ifhmode\unskip\space\fi MR }
\providecommand{\MRhref}[2]{%
  \href{http://www.ams.org/mathscinet-getitem?mr=#1}{#2}
}
\providecommand{\href}[2]{#2}
\begin{thebibliography}{10}

\bibitem{rubin-short}
James Belk, Luna Elliott, and Francesco Matucci, \emph{A short proof of
  {R}ubin's theorem}, Israel J. Math. \textbf{267} (2025), no.~1, 157--169.
  \MR{4930683}

\bibitem{GKL2023}
Isaac Goldbring, Srivatsav Kunnawalkam~Elayavalli, and Yash Lodha,
  \emph{Generic algebraic properties in spaces of enumerated groups}, Trans.
  Amer. Math. Soc. \textbf{376} (2023), no.~9, 6245--6282. \MR{4630775}

\bibitem{hodges-book}
Wilfrid Hodges, \emph{Building models by games}, London Mathematical Society
  Student Texts, vol.~2, Cambridge University Press, Cambridge, 1985.
  \MR{812274}

\bibitem{kaplansky-book}
Irving Kaplansky, \emph{Infinite abelian groups}, University of Michigan Press,
  Ann Arbor, MI, 1954. \MR{65561}

\bibitem{KK2021book}
Sang-hyun Kim and Thomas Koberda, \emph{Structure and regularity of group
  actions on one-manifolds}, Springer Monographs in Mathematics, Springer,
  Cham, [2021] \copyright 2021. \MR{4381312}

\bibitem{KKdlN22}
Sang-hyun Kim, Thomas Koberda, and J.~de~la Nuez~Gonz\'alez, \emph{First order
  rigidity of homeomorphism groups of manifolds}, Commun. Am. Math. Soc.
  \textbf{5} (2025), 144--194. \MR{4904300}

\bibitem{KdlN2023}
Thomas Koberda and J.~de~la Nuez~González, \emph{Uniform first order
  interpretation of the second order theory of countable groups of
  homeomorphisms}, 2023, arXiv:2312.16334.

\bibitem{KdlN2024}
\bysame, \emph{Locally approximating groups of homeomorphisms of manifolds},
  2024, arXiv:2410.16108.

\bibitem{LS2001}
R.~C. Lyndon and P.~E. Schupp, \emph{Combinatorial group theory}, Classics in
  Mathematics, Springer-Verlag, Berlin, 2001, Reprint of the 1977 edition.
  \MR{1812024 (2001i:20064)}

\bibitem{marker-infinitary}
David Marker, \emph{Lectures on infinitary model theory}, Lecture Notes in
  Logic, vol.~46, Association for Symbolic Logic, Chicago, IL; Cambridge
  University Press, Cambridge, 2016. \MR{3558585}

\bibitem{olshanski-book}
A.~Yu. Ol\cprime~shanski\u i, \emph{Geometry of defining relations in groups},
  Mathematics and its Applications (Soviet Series), vol.~70, Kluwer Academic
  Publishers Group, Dordrecht, 1991, Translated from the 1989 Russian original
  by Yu.\ A. Bakhturin. \MR{1191619}

\bibitem{Rubin1989}
Matatyahu Rubin, \emph{On the reconstruction of topological spaces from their
  groups of homeomorphisms}, Trans. Amer. Math. Soc. \textbf{312} (1989),
  no.~2, 487--538. \MR{988881}

\bibitem{Rubin1996}
\bysame, \emph{Locally moving groups and reconstruction problems}, Ordered
  groups and infinite permutation groups, Math. Appl., vol. 354, Kluwer Acad.
  Publ., Dordrecht, 1996, pp.~121--157. \MR{1486199}

\bibitem{Serre1977}
J.-P. Serre, \emph{Arbres, amalgames, {${\rm SL}_{2}$}}, Soci\'et\'e
  Math\'ematique de France, Paris, 1977, Avec un sommaire anglais, R\'edig\'e
  avec la collaboration de Hyman Bass, Ast\'erisque, No. 46. \MR{0476875}

\end{thebibliography}
\end{document}